\documentclass[fontsize=11pt,paper=a4,reqno,numbers=noendperiod,bibliography=totoc]{scrartcl}
\usepackage[greek,USenglish]{babel}
\usepackage[T1]{fontenc}
\usepackage[latin10]{inputenc}
\usepackage{amsmath,amsthm} 
\usepackage{fouriernc}
\usepackage{dsfont,upgreek}
\usepackage[obeyspaces]{url}
\usepackage{colortbl,color,xcolor}
\usepackage{graphicx,tikz,pgfplots}
\usepackage{centernot}
\usepackage{array}
\usepackage{enumitem}
\usepackage[bookmarks=true,plainpages=false,linktocpage,colorlinks=true,citecolor=green!80!black,linkcolor=red!70!black,filecolor=magenta,urlcolor=magenta,breaklinks,pdfauthor={Thomas Jahn},pdftitle={Extremal Radii, Diameter, and Minimum Width in Generalized Minkowski Spaces}]{hyperref}
\usepackage[babel]{microtype}
\usepackage{subfigure}
\usepackage{scrhack}
\usepackage{float}
\usepackage{cite}

\usetikzlibrary{calc,intersections,through,backgrounds,arrows,patterns,shapes.geometric,shapes.misc}

\newcommand{\KK}{\mathcal{K}}
\newcommand{\CK}{\mathcal{C}}
\newcommand{\LL}{\mathcal{L}}
\newcommand{\NN}{\mathds{N}}
\newcommand{\RR}{\mathds{R}}

\newcommand{\abs}[1]{\left\lvert#1\right\rvert}          
\newcommand{\norm}[1]{\left\lvert#1\right\rvert}         
\newcommand{\setn}[1]{\left\{#1\right\}}                 
\newcommand{\setns}[1]{\{#1\}}                           
\newcommand{\setcond}[2]{\left\{#1 \:\middle\vert\: #2\right\}}    
\newcommand{\defeq}{\mathrel{\mathop:}=}
\newcommand{\zB}{e.g., }
\newcommand{\dah}{i.e., }
\newcommand{\lr}[1]{\!\left(#1\right)}

\newcommand{\cRR}{\overline{\mathds{R}}}
\newcommand{\clseg}[2]{\left[#1,#2\right]}
\newcommand{\skpr}[2]{\left\langle#1 \,\middle\vert\, #2\right\rangle}
\newcommand{\enquote}[1]{``#1''}
\newcommand{\norel}{\mathrel{\phantom{=}}}

\theoremstyle{plain}
\newtheorem{Satz}{Theorem}[section]

\newtheorem{Lem}[Satz]{Lemma}
\newtheorem{Prop}[Satz]{Proposition}

\theoremstyle{definition}
\newtheorem{Def}[Satz]{Definition}
\newtheorem{Bsp}[Satz]{Example}
\newtheorem{Bem}[Satz]{Remark}

\DeclareMathOperator{\co}{co}
\DeclareMathOperator{\cl}{cl}
\DeclareMathOperator{\inte}{int}
\DeclareMathOperator{\bd}{bd}
\DeclareMathOperator{\ri}{relint}
\DeclareMathOperator{\dist}{dist}

\let \subset \subseteq
\let \supset \supseteq
\newcommand{\subsetneq}{\varsubsetneq}

\addtokomafont{caption}{\small}
\setkomafont{captionlabel}{\bfseries}

\begin{document}
\parindent 0pt
\title{Extremal Radii, Diameter, and Minimum Width in Generalized Minkowski Spaces}
\author{Thomas Jahn\\
{\small Faculty of Mathematics, Technische Universit\"at Chemnitz}\\
{\small 09107 Chemnitz, Germany}\\
{\small thomas.jahn\raisebox{-1.5pt}{@}mathematik.tu-chemnitz.de}
}
\date{}
\maketitle
\allowdisplaybreaks[2]

\begin{abstract}
We discuss the notions of circumradius, inradius, diameter, and minimum width in generalized Minkowski spaces (that is, with respect to gauges), \dah we measure the \enquote{size} of a given convex set in a finite-dimensional real vector space with respect to another convex set. This is done via formulating some kind of containment problem incorporating homothetic bodies of the latter set or strips bounded by parallel supporting hyperplanes thereof. The paper can be seen as a theoretical starting point for studying metrical problems of convex sets in generalized Minkowski spaces.
\end{abstract}



\section{Introduction}\label{chap:introduction}

The celebrated \emph{Sylvester problem}, which was originally posed in \cite{Sylvester1857}, asks for a point that minimizes the maximum distance to points from a given finite set in the Euclidean plane. There are at least two ways to generalize this problem. From a first point of view, we might keep the participating geometric configuration --~given a set, we are searching a point~-- but change the distance measurement. Classically, distance measurement is provided by the Euclidean norm or, equivalently, by its unit ball, which is a centered, compact, convex set having the origin as interior point.
Then the Sylvester problem asks for the least scaling factor (called circumradius) such that there is a correspondingly scaled version of the unit ball that contains the given set. In the literature \cite{BrandenbergKoe2014}, this setting has already been relaxed by using norms \cite{AlonsoMaSp2012b,Jahn2015a,MartinMaSp2014} and even by dropping the centeredness and the boundedness of the unit ball as well as the finite cardinality of the given set \cite{BrandenbergKoe2013,BrandenbergKoe2014,BrandenbergRo2011}. Vector spaces equipped with such a unit ball shall be called \emph{generalized Minkowski spaces}.
The corresponding analogue of the norm is the Minkowski functional of the unit ball, which is also called \emph{gauge} or \emph{convex distance function} in the literature. A second possibility to change the setting of the Sylvester problem is as follows. We keep the Euclidean distance measurement, but instead of asking for a point which approximates the given set in a minimax sense, we ask for an affine flat of certain dimension doing this. In this paper, we focus on generalizing the distance measurement and, after obtaining an appropriate notion of circumradius, discuss how to define the notions of inradius, diameter, and minimum width within the general setting of generalized Minkowski spaces. The second way of generalizing the Sylvester problem (namely by involving affine flats of certain dimension) is investigated in \cite{Jahn2015c}.

The paper is organized as follows. In Section~\ref{chap:preliminaries}, we introduce our notation and recall some basic facts regarding support functions and width functions. Results on the four classical quantities circumradius, inradius, diameter, and minimum width are presented in Section~\ref{chap:classic}. The paper is finished by a collection of open questions in Section~\ref{chap:open}.

\section{Preliminaries}\label{chap:preliminaries}
Four classical quantities for measuring the size of a given set are: the maximum distance between two of its points (its \emph{diameter}), the minimal distance between two parallel supporting hyperplanes (its \emph{minimum width} or \emph{thickness}), the radius of the smallest ball containing the set (its \emph{circumradius}), and the radius of the largest ball that is contained in the set (its \emph{inradius}).
In the framework of convex geometry, the definitions of these quantities refer to Euclidean distance measurement, that is, we compare the size of the given set with the size of the Euclidean unit ball. In the following, we will describe how diameter, minimum width, circumradius, and inradius can be defined precisely, when comparing sizes with a centered convex body (not necessarily the Euclidean unit ball), and what can be done when we even drop the centeredness of the measurement body. At first, we have a look at support and width functions, which, for convex sets, are related to some kind of signed Euclidean distances between supporting hyperplanes and the origin and between parallel supporting hyperplanes, respectively.

Throughout this paper, we shall be concerned with the vector space $\RR^d$, with the topology generated by the usual inner product $\skpr{\cdot}{\cdot}$ and the norm $\norm{\cdot}=\sqrt{\skpr{\cdot}{\cdot}}$ or, equivalently, its unit ball $B$. For the \emph{extended real line}, we write $\cRR\defeq \RR\cup\setn{+\infty,-\infty}$ with the conventions $0(+\infty)\defeq +\infty$, $0(-\infty)\defeq 0$ and $(+\infty)+(-\infty)\defeq +\infty$.
We use the notation $\CK^d$ for the family of \emph{non-empty closed convex sets} in $\RR^d$. We denote the class of bounded sets that belong to $\CK^d$ by $\KK^d$. We also write $\CK^d_0$ and $\KK^d_0$ for the classes of sets having non-empty interior and belonging to $\CK^d$ and $\KK^d$, respectively. The \emph{line segment} between $x$ and $y$ shall be denoted by $\clseg{x}{y}$. The abbreviations $\cl$, $\inte$ and $\co$ stand for \emph{closure}, \emph{interior} and \emph{convex hull}, respectively. A set $K$ is \emph{centrally symmetric} iff there is a point $z\in\RR^d$ such that $K=2z-K$, and $K$ is said to be \emph{centered} iff $K=-K$.

\begin{Def}\label{def:support_function}
The \emph{support function} of a set $K\subset\RR^d$ is defined as $h_K:\RR^d\to\cRR$, $h_K(x)\defeq\sup\setcond{\skpr{x}{y}}{y\in K}$. Its sublevel set
\begin{equation*}
K^\circ\defeq \setcond{x\in\RR^d}{h_K(x)\leq 1}
\end{equation*}
is called the \emph{polar set} of $K$.
\end{Def}

\begin{Lem}[{\cite[Proposition~7.11]{BauschkeCo2011}, \cite[\textsection~15]{BonnesenFe1987}}]\label{lem:support_function}
Let $K,K^\prime\subset \RR^d$, $x,y\in\RR^d$, and $\alpha >0$. We have
\begin{enumerate}[label={(\alph*)},align=left]
\item{$h_K=h_{\cl(K)}=h_{\co(K)}$,\label{support_function_closed_convex}}
\item{$h_{K+K^\prime}=h_K+h_{K^\prime}$,\label{support_function_addition}}
\item{sublinearity: $h_K(x+y)\leq h_K(x)+h_K(y)$, $h_K(\alpha x)=\alpha h_K(x)$,\label{support_function_sublinear}}
\item{$h_{\alpha K}(x)=\alpha h_K(x)$, $h_{-K}(x)=h_K(-x)$.\label{support_function_scaling}}
\end{enumerate}
\end{Lem}

The previous lemma tells us that it suffices to consider closed and convex sets for the study of support functions.  As mentioned before, there is a link between support functions and Euclidean distances between the origin and supporting hyperplanes. These are given by 
\begin{equation*}
H_K(u)=\setcond{y\in\RR^d}{\skpr{u}{y}=h_K(u)}
\end{equation*}
(provided $h_K(u)<+\infty$), and $u$ is then called the \emph{outer normal vector} of the supporting hyperplane. The \emph{Euclidean distance function} of a set $K$ evaluated at $x\in\RR^d$ is given via minimal distances, namely $\dist(x,K)\defeq \inf\setcond{\norm{y-x}}{y\in K}$.
\begin{Prop}[{\cite[Theorem~1.1]{PlastriaCa2001}}]\label{prop:support_hyperplane_distance}
Let $K\in\CK^d$ and $u\in\RR^d$ be such that $\norm{u}=1$. We have
\begin{equation*}
\abs{h_K(u)}=\dist(0,H_K(u)).
\end{equation*}
\end{Prop}

The distances between parallel supporting hyperplanes of a set $K$ are encoded by its width function. 
\begin{Def}\label{def:width_function}
The \emph{width function} of a set $K\subset \RR^d$ is defined as $w_K:\RR^d \to \cRR$, $w_K(x)\defeq h_K(x)+h_K(-x)$.
\end{Def}
\begin{Lem}\label{lem:width_function}
Let $K,K^\prime\subset \RR^d$, $u\in\RR^d$, and $\alpha>0$. We have
\begin{enumerate}[label={(\alph*)},align=left]
\item{$w_K=h_{K-K}$,\label{width_function_support}}
\item{$w_K=w_{\co(K)}=w_{\cl(K)}$,\label{width_function_closed_convex}}
\item{$w_K$ is sublinear and non-negative,\label{width_function_sublinear}}
\item{$w_{K+K^\prime}=w_K+w_{K^\prime}$,\label{width_function_addition}}
\item{$w_{\alpha K}=\alpha w_K$, $w_{-K}=w_K$,\label{width_function_scaling}}
\item{if $w_K(u)<+\infty$, then $w_K(u)=\dist(0,H_K(u))+\dist(0,H_K(-u))=\dist(y,H_K(-u))$ for all $y\in H_K(u)$.\label{width_function_distance}}
\end{enumerate}
\end{Lem}

\section{The four classical quantities}\label{chap:classic}
\subsection{Circumradius: measuring from outside}
The definition of the circumradius can be found, \zB in \cite{GonzalezHe2012,GonzalezHeHi2015,HenkHe2009} for the case $C=B$ (Euclidean space), and in \cite{GritzmannKl1992,MartinMaSp2014} for the case $C=-C\in \KK^d_0$ (normed spaces).
\begin{Def}[{\cite{BrandenbergRo2011,BrandenbergKoe2013,BrandenbergKoe2014}}]\label{def:circumradius}
The \emph{circumradius} of $K\subset \RR^d$ with respect to $C\in\CK^d$ is defined as
\begin{equation*}
R(K,C)=\inf_{x\in\RR^d}\inf\setcond{\lambda>0}{K\subset x+\lambda C}.
\end{equation*}
If $K\subset x+R(K,C)C$, then $x$ is a \emph{circumcenter} of $K$ with respect to $C$.
\end{Def}
\begin{figure}[h!]
\begin{center}
\begin{tikzpicture}[line cap=round,line join=round,>=stealth,x=0.7cm,y=0.7cm]
\pgfmathsetmacro{\a}{sqrt(3)}
\draw [line width=1.2pt](-1,\a) arc (150:210:2*\a) -- (-1,-\a) arc (270:330:2*\a) -- (2,0) arc (30:90:2*\a)--cycle;
\draw [shift={(5,0)},scale=0.65](-1,\a) arc (150:210:2*\a) -- (-1,-\a) arc (270:330:2*\a) -- (2,0) arc (30:90:2*\a)--cycle;
\draw [line width=1.2pt,shift={(5,0)},scale=0.65]($(-1,-\a)+(50:2*\a)$)--($(-1,\a)+(300:2*\a)$)--($(2,0)+(200:2*\a)$)--cycle;
\end{tikzpicture}
\end{center}\caption{Circumradius: The set $C$ is a Reuleaux triangle (bold line, left), the set $K$ is a triangle (bold line, right). The circumradius $R(K,C)$ is determined by the smallest homothet of $C$ that contains $K$ (thin line).}
\end{figure}
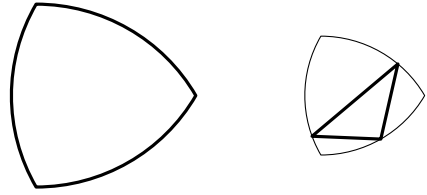
\begin{Prop}\label{prop:circumradius}
Let $K,K^\prime\subset \RR^d$, $C,C^\prime\in\CK^d$, and $\alpha,\beta>0$. Then
\begin{enumerate}[label={(\alph*)},align=left]
\item{$R(K^\prime,C^\prime)\leq R(K,C)$ if $K^\prime \subset K$ and $C\subset C^\prime$,\label{circumradius_inclusion}}
\item{$R(K,C)=R(\cl(K),C)=R(\co(K),C)$,\label{circumradius_closed_convex}}
\item{$R(K+K^\prime,C)\leq R(K,C)+R(K^\prime,C)$,\label{circumradius_addition}}
\item{$R(x+K,y+C)=R(K,C)$ for all $x,y\in\RR^d$,\label{circumradius_translation}}
\item{$R(\alpha K,\beta C)=\frac{\alpha}{\beta}R(K,C)$,\label{circumradius_scaling}}
\item{$R(K,C^\prime)\leq R(K,C)R(C,C^\prime)$.\label{circumradius_nesting}}
\end{enumerate}
\end{Prop}
\begin{proof}
For $x\in\RR^d$ and $\lambda\geq 0$, we have $\cl(K)\subset x+\lambda C\Longleftrightarrow K\subset x+\lambda C\Longleftrightarrow \co(K)\subset x+\lambda C$. This proves \ref{circumradius_closed_convex}.
Statement~\ref{circumradius_addition} is also rather simple: If $K\subset z+\lambda C$ and $K^\prime\subset z^\prime+\lambda^\prime C$ for some $z,z^\prime\in\RR^d$, $\lambda,\lambda^\prime>0$, then $K+K^\prime\subset (z+z^\prime)+(\lambda+\lambda^\prime)C$. In order to show \ref{circumradius_nesting}, note that there exist numbers $\lambda, \lambda^\prime>0$ and points $z,z^\prime\in \RR^d$ such that $K\subset z+\lambda C$ and $C\subset z^\prime+\lambda^\prime C^\prime$. Substituting the latter inclusion into the former one, we obtain $K\subset z+\lambda z^\prime+\lambda\lambda^\prime C^\prime$.
\end{proof}
\begin{Bem}
\begin{enumerate}[label={(\alph*)},align=left]
\item{The following implication is wrong: If $K,K^\prime\in\CK^d$ and $C\in\KK^d_0$, then $R(K+K^\prime,C+K^\prime)=R(K,C)$. For example, take 
\begin{equation*}
C=\clseg{-e_1}{e_1}+\ldots+\clseg{-e_d}{e_d},\qquad K=\co\setn{0,e_1,\ldots,e_d},\qquad K^\prime =\clseg{0}{e_d},
\end{equation*}
where $e_i\in\RR^d$ denotes the vector whose entries are $0$ except for the $i$th one, which is $1$. Then $R(K,C)=\frac{1}{2}$ and $R(K+K^\prime,C+K^\prime)=\frac{2}{3}$. But the following implication is true: If $K,C^\prime\in\CK^d$, $K\in\KK^d$, and $R(K,C)=1$, then $R(K+K^\prime,C+K^\prime)=R(K,C)$. This is because of the cancellation rule, which reads as
\begin{center}
\emph{\enquote{Let $K_1,K_2\in \CK^d$ and $K\in \KK^d$. If $K_1+K\subset K_2+K$, then $K_1\subset K_2$.}}
\end{center}
and can easily be proved via support functions. For all $\lambda>1$, there exists $z\in\RR^d$ such that $K\subset z+\lambda C$. It follows that $K+K^\prime \subset z+\lambda C+K^\prime \subset z+\lambda (C+K^\prime)$. In other words, $R(K+K^\prime,C+K^\prime)\leq 1$. Conversely, assume that $R(K+K^\prime,C+K^\prime)< 1$. Then there are $\lambda< 1$ and $z\in\RR^d$ such that $K+K^\prime\subset z+\lambda (C+K^\prime)\subset z+\lambda C+K^\prime$. By virtue of the cancellation rule, we have $K\subset z+\lambda (C+K^\prime)\subset z+\lambda C$, which is a contradiction to $R(K,C)=1$.}
\item{Proposition~\ref{prop:circumradius}\ref{circumradius_addition} holds with equality if $K^\prime=\alpha C$ for all $\alpha>0$.}
\end{enumerate}
\end{Bem}
The previous lemma tells us that the circumradius of $K$ with respect to $C$ is invariant under translations of both $K$ and $C$. So, without loss of generality, we may assume that $0\in\ri(C)$. Then the circumradius can be equivalently written as 
\begin{equation}
R(K,C)=\inf_{x\in\RR^d}\sup_{y\in K}\gamma_C(y-x),\label{eq:circumradius_convex_program}
\end{equation}
where $\gamma_C:\RR^d\to\cRR$ is the \emph{Minkowski functional} defined by $\gamma_C(x)\defeq \inf\setcond{\lambda>0}{x\in\lambda C}$.

\begin{Lem}\label{lem:segment_length}
Given $x,y,z\in\RR^d$, $\alpha >0$, and $C\in\KK^d_0$ with $0\in\inte(C)$, we define $g(x,y)\defeq 2R(\setn{x,y},C)$. The following statements are true:
\begin{enumerate}[label={(\alph*)},align=left]
\item{$g(x,y)\geq 0$ with equality if and only if $x=y$,}
\item{$g(x,y)=g(y,x)$,}
\item{$g(x+z,y+z)=g(x,y)$,}
\item{$g(\alpha x,\alpha y)=\alpha g(x,y)$,}
\item{$g(x,y)\leq g(x,w)+g(w,y)$.}
\end{enumerate}
\end{Lem}
\begin{proof}
The non-negativity follows from the definition. The characterization of the equality case is a consequence of the more general result that $R(K,C)=0$ if and only if $K$ is contained in a translate of the cone $\setcond{y\in\RR^d}{y+C\subset C}$ when the set of extreme points of $C$ is bounded, see \cite[Lemma~2.2]{BrandenbergKoe2014}. The symmetry $g(x,y)=g(y,x)$ is clear. The invariance under translations and the compatibility with scaling follow from Proposition~\ref{prop:circumradius}\ref{circumradius_translation}, \ref{circumradius_scaling}. Finally, since $g$ is translation-invariant and symmetric, we only have to check $g(0,x+y)\leq g(0,x)+g(0,y)$ for the triangle inequality. But we have
\begin{align*}
g(0,x+y)&= R(\setn{0,x+y},C)\leq R(\setn{0,x,y,x+y},C)\\
&\leq R(\setn{0,x},C)+R(\setn{0,y},C)=g(0,x)+g(0,y)
\end{align*}
by Proposition~\ref{prop:circumradius}\ref{circumradius_addition}.
\end{proof}
Since the triangle inequality for $g$ turns out to be true, the mapping $x\mapsto 2 R(\setn{0,x},C)$ defines a norm on $\RR^d$. The unit ball of this norm is $\frac{1}{2}(C-C)$. This fact can be proved as follows. First, we show that if $x\in \frac{1}{2}(C-C)$, then $R(\setn{0,x},\frac{1}{2}C)\leq 1$. Namely, there exist $y_1,y_2\in \frac{1}{2}C$ such that $x=y_1-y_2$. Thus $R(\setn{0,x},\frac{1}{2}C)=R(\setn{y_1,y_2},\frac{1}{2}C)\leq 1$. The reverse implication is as easy as the first one. If $R(\setn{0,x},\frac{1}{2}C)> 1$, then there is no point $z\in\RR^d$ such that $\setn{0,x}\in z+\frac{1}{2}C$ or, equivalently, such that $\setn{-z,x-z}\in \frac{1}{2}C$. Thus there is no representation $x=(x-z)-(-z)\in \frac{1}{2}C-\frac{1}{2}C$.
\begin{Def}\label{def:max_chord_length}\label{def:radius_function}
The \emph{maximal chord-length function} of $K\subset \RR^d$ is defined as $l_K:\RR^d\to \cRR$,
\begin{equation*}
l_K(x)=\sup\setcond{\alpha >0}{\alpha x\in K-K}.
\end{equation*}
The \emph{radius function} $r_K:\RR^d \to \cRR$, defined as
\begin{equation*}
r_K(u)=\sup\setcond{\alpha >0}{\alpha u \in K},
\end{equation*}
is the pointwise inverse to the Minkowski functional $\gamma_K$.
\end{Def}
\begin{Lem}
Let $K\subset \RR^d$, $x,y\in \RR^d$, and $\beta_1,\beta_2>0$. We have
\begin{enumerate}[label={(\alph*)},align=left]
\item{$l_{\beta_2 K}(\beta_1 x)= \frac{\beta_2}{\beta_1}l_K(x)$,}
\item{$l_{y+K}(x)=l_K(-x)=l_{-K}(x)=l_K(x)$.}
\end{enumerate}
\end{Lem}
\begin{proof}
We only prove the first part, because the second one is an easy consequence of the centeredness of $K-K$. We have
\begin{align*}
l_{\beta_2 K}(\beta_1 x)&=\sup\setcond{\alpha >0}{\alpha \beta_1 x\in \beta_2 K}=\sup\setcond{\alpha >0}{\alpha \frac{\beta_1}{\beta_2} x\in K}\\
&=\sup\setcond{\gamma \frac{\beta_2}{\beta_1}}{\gamma>0, \gamma x\in K}=\frac{\beta_2}{\beta_1}l_K(x).\qedhere
\end{align*}
\end{proof}
For centered sets $K$, the maximal circumradius of two-element subsets is attained at antipodal points of $K$.
\begin{Prop}
Let $K\subset\RR^d$ be a bounded set, and let $C\in \KK^d_0$. If $K=-K$, then
\begin{equation*}
\sup\setcond{R(\setn{-x,x},C)}{x\in K}=\sup\setcond{R(\setn{x,y},C)}{x,y\in K}.
\end{equation*}
\end{Prop}
\begin{proof}
Using Proposition~\ref{prop:circumradius}, we have
\begin{align*}
&\norel\sup\setcond{R(\setn{x,y},C)}{x,y\in K}\\
&\leq\sup\setcond{R(\setn{0,x,y,x+y},C)}{x\in K}\\
&\leq\sup\setcond{R(\setn{0,x},C)}{x\in K}+\sup\setcond{R(\setn{0,y},C)}{y\in K}\\
&=2\sup\setcond{R(\setn{0,x},C)}{x\in K}\\
&=\sup\setcond{R(\setn{0,2x},C)}{x\in K}\\
&=\sup\setcond{R(\setn{-x,x},C)}{x\in K}\\
&\leq\sup\setcond{R(\setn{x,y},C)}{x,y\in K}.\qedhere
\end{align*}
\end{proof}
Similarly, the maximum chord length of centered convex bounded sets is attained at antipodal points of $K$.
\begin{Lem}\label{lem:maximum_chord_symmetric}
Let $K\in \KK^d$, $K=-K$, and $u\in\RR^d$ be such that $\norm{u}=1$. Then there is $z\in K$ such that $l_K(u)=\norm{z-(-z)}=2\norm{z}$.
\end{Lem}
\begin{proof}
We have
\begin{align}
l_K(u)&=\sup\setcond{\alpha >0}{\alpha u \in K-K}=\sup\setcond{\alpha >0}{\alpha u \in 2K}\nonumber\\
&=\sup\setcond{\alpha >0}{\frac{1}{2}\alpha u \in K}=2\sup\setcond{\alpha >0}{\alpha u \in K}.\label{eq:maximum_chord_symmetric}
\end{align}
Since $K$ is compact, we have $\sup\setcond{\alpha >0}{\alpha u \in K}<+\infty$, and therefore
\begin{equation*}
z\defeq\sup\setcond{\alpha >0}{\alpha u \in K}u\in K.\qedhere
\end{equation*}
\end{proof}

If both $K$ and $C$ are centered, there is another nice representation of the circumradius.
\begin{Prop}[{\cite[(1.1)]{GritzmannKl1992}}]\label{lem:central_symmetry_at_all}
Let $K\subset\RR^d$, $C\in \CK^d_0$, $0\in \inte(C)$, $C=-C$, $K=-K$. Then
\begin{equation*}
R(K,C)=\sup\setcond{\gamma_C(x)}{x\in K}.
\end{equation*}
\end{Prop}
\begin{proof}
If $K\subset z+\lambda C$ for suitable $z\in \RR^d$ and $\lambda>0$, then $K\subset -z+\lambda C$ due to the centeredness of $K$ and $C$. It follows that
\begin{equation*}
K\subset \frac{1}{2}K+\frac{1}{2}K\subset \frac{1}{2}(z+\lambda C)+\frac{1}{2}(-z+\lambda C) = \lambda C.
\end{equation*}
In other words, the circumradius is already determined by the sets $\lambda C$ with $\lambda >0$:
\begin{equation*}
R(K,C)=\inf\setcond{\lambda >0}{K\subset \lambda C}=\sup\setcond{\gamma_C(x)}{x\in K}.\qedhere
\end{equation*}
\end{proof}

\subsection{Inradius: measuring from inside}\label{chap:inradius}
The definition of the inradius can be found, \zB in \cite{GonzalezHe2012,GonzalezHeHi2015,HenkHe2009} for the case $C=B$ (Euclidean space) and in \cite{GritzmannKl1992} for the case $C=-C\in \KK^d_0$ (normed spaces).
\begin{Def}\label{def:inradius}
The \emph{inradius} of $K\subset \RR^d$ with respect to $C\in\CK^d$ is defined as
\begin{equation*}
r(K,C)=\sup_{x\in\RR^d}\sup\setcond{\lambda\geq 0}{x+\lambda C \subset K}.
\end{equation*}
\end{Def}
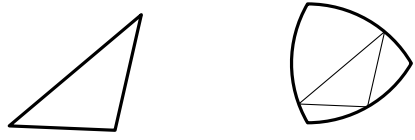
\begin{figure}[h!]
\begin{center}
\begin{tikzpicture}[line cap=round,line join=round,>=stealth,x=0.7cm,y=0.7cm]
\pgfmathsetmacro{\a}{sqrt(3)}
\draw [line width=1.2pt]($(-1,-\a)+(50:2*\a)$)--($(-1,\a)+(300:2*\a)$)--($(2,0)+(200:2*\a)$)--cycle;
\draw [line width=1.2pt,shift={(5,0)},scale=0.65](-1,\a) arc (150:210:2*\a) -- (-1,-\a) arc (270:330:2*\a) -- (2,0) arc (30:90:2*\a)--cycle;
\draw [shift={(5,0)},scale=0.65]($(-1,-\a)+(50:2*\a)$)--($(-1,\a)+(300:2*\a)$)--($(2,0)+(200:2*\a)$)--cycle;
\end{tikzpicture}
\end{center}\caption{Inradius: The set $C$ is a triangle (bold line, left), the set $K$ is a Reuleaux triangle (bold line, right). The circumradius $R(K,C)$ is determined by the largest homothet of $C$ that is contained in $K$ (thin line).}
\end{figure}
This definition is similar to the definition of the circumradius, and so are the corresponding basic properties.
\begin{Prop}\label{prop:inradius}
Let $K,K^\prime\subset \RR^d$, $C,C^\prime\in\CK^d$, and $\alpha,\beta>0$. Then
\begin{enumerate}[label={(\alph*)},align=left]
\item{$r(K^\prime,C^\prime)\geq r(K,C)$ if $K^\prime \subset K$ and $C\subset C^\prime$,\label{inradius_inclusion}}
\item{$r(K,C)=r(\cl(K),C)$ if $K$ is convex,\label{inradius_closed_convex}}
\item{$r(K+K^\prime,C)\geq r(K,C)+r(K^\prime,C)$,\label{inradius_addition}}
\item{$r(x+K,y+C)=r(K,C)$ for all $x,y\in\RR^d$,\label{inradius_translation}}
\item{$r(\alpha K,\beta C)=\frac{\alpha}{\beta}r(K,C)$,\label{inradius_scaling}}
\item{$r(K,C^\prime)\geq r(K,C)r(C,C^\prime)$.\label{inradius_nesting}}
\end{enumerate}
\end{Prop}
\begin{proof}
For $x\in\RR^d$ and $\lambda\geq 0$, we have $\cl(K)\subset x+\lambda C\Longleftrightarrow K\subset x+\lambda C$. This proves \ref{inradius_closed_convex}.
Statement~\ref{inradius_addition} is rather simple: If $K\supset z+\lambda C$ and $K^\prime\supset z^\prime+\lambda^\prime C$ for some $z,z^\prime\in\RR^d$, $\lambda,\lambda^\prime>0$, then $K+K^\prime\supset (z+z^\prime)+(\lambda+\lambda^\prime)C$. In order to show \ref{inradius_nesting}, note that there exist numbers $\lambda, \lambda^\prime>0$ and points $z,z^\prime\in \RR^d$ such that $K\supset z+\lambda C$ and $C\supset z^\prime+\lambda^\prime C^\prime$. Substituting the latter inclusion into the former one, we obtain $z+\lambda z^\prime+\lambda\lambda^\prime C^\prime \subset K$.
\end{proof}

\subsection{Diameter}
In Euclidean geometry, the \emph{diameter} of a given set is usually defined as the maximum distance of two points of this set. But there are several other representations of this quantity which do not coincide when replacing the Euclidean unit ball by a convex body $C$ in general (but at least if $C=-C$). This offers various possibilities to think about an appriopriate extension of the notion of diameter. At first, let us consider the interpretation of the diameter as maximum distance between points of the set. Here, the distance notion is provided by the Minkowski functional of $C$. Then we can rewrite the expression for the diameter as the supremum of the Euclidean width function over the polar set of $C$.
\begin{Satz}\label{thm:diameter_max_distance}
For $K\subset \RR^d$ and $C\in \KK^d_0$ with $0\in\inte(C)$, the following numbers are equal:
\begin{enumerate}[label={(\alph*)},align=left,series=diameter1]
\item{$\sup\setcond{\gamma_C(x-y)}{x,y\in K}$,\label{max_distance_diam}}
\item{$\sup\setcond{w_K(u)}{u\in C^\circ}$,}
\item{$\sup\setcond{\skpr{u}{x}}{u\in C^\circ, x\in K-K}$.}
\end{enumerate}
If $K\in \CK^d$, then the following number also belongs to this set of equal quantities:
\begin{enumerate}[label={(\alph*)},align=left,diameter1]
\item{$\sup\setcond{\frac{l_K(u)}{r_C(u)}}{u\in \RR^d\setminus\setn{0}}$.}
\end{enumerate}
\end{Satz}
\begin{proof}
Using \cite[Lemma~2.1]{JahnKuMaRi2014}, we have

\begin{align*}
&\norel\sup_{x,y\in K}\gamma_C(x-y)\\
&=\sup_{x,y\in K}\sup_{u\in C^\circ}\skpr{u}{x-y}\\
&=\sup_{u\in C^\circ}\sup_{x,y\in K}\skpr{u}{x-y}\\
&=\sup_{u\in C^\circ}\sup_{x,y\in K}\lr{\skpr{u}{x}+\skpr{-u}{y}}\\
&=\sup_{u\in C^\circ}\lr{h_K(u)+h_K(-u)}\\
&=\sup_{u\in C^\circ}w_K(u)\\
&=\sup_{u\in C^\circ}h_{K-K}(u)\\
&=\sup\setcond{\skpr{u}{x}}{u\in C^\circ, x\in K-K}.
\end{align*}

If $K\in \CK^d$, then
\begin{align*}
\sup_{x,y\in K} \gamma_C(x-y)&=\sup_{u\in \RR^d\setminus\setn{0}}\sup_{\alpha >0:\,\alpha u\in K-K} \gamma_C(\alpha u)\\
&=\sup_{u\in \RR^d\setminus\setn{0}}\sup_{\alpha >0:\,\alpha u\in K-K}\alpha  \gamma_C(u)\\
&=\sup_{u\in \RR^d\setminus\setn{0}}l_K(u) \gamma_C(u)\\
&=\sup_{u\in \RR^d\setminus\setn{0}}\frac{l_K(u)}{r_C(u)}.
\end{align*}
\end{proof}
Other representations of the diameter in the Euclidean case are written in terms of circumradii, see \cite[Theorem~2]{Averkov2003b}. Together with the representation from Theorem \ref{thm:diameter_max_distance}, we obtain a chain of inequalities.
\begin{Satz}\label{thm:symmetric_diameter}
If $K\subset \RR^d$ and $C\in \KK^d_0$ with $0\in\inte(C)$, then
\begin{equation}
\left.\hspace{4cm}\begin{aligned}
&\norel 2\sup\setcond{\frac{h_{K-K}(u)}{h_{C-C}(u)}}{u\in \RR^d\setminus\setn{0}}\\
&= 2\sup\setcond{R(\setn{x,y},C)}{x,y\in K}\\
&= R\lr{K-K,\frac{1}{2}(C-C)}\\
&\leq R(K-K,C)\\
&\leq \sup\setcond{\gamma_C(x-y)}{x,y\in K},
\end{aligned}\hspace{4cm}\right\}\label{eq:symmetric_diameter}
\end{equation}
with equality if $C=-C$. If $K\in \CK^d$, then we have also
\begin{equation}
\sup\setcond{R(\setn{x,y},C)}{x,y\in K}=\sup\setcond{\frac{l_K(u)}{l_C(u)}}{u\in \RR^d\setminus\setn{0}}.\label{eq:symmetric_diameter_addition}
\end{equation}
\end{Satz}
\begin{proof}
If $C=-C$, then we have
\begin{align*}
&\norel 2\sup\setcond{\frac{h_{K-K}(u)}{h_{C-C}(u)}}{u\in \RR^d\setminus\setn{0}}\\
&=\sup\setcond{\frac{h_{K-K}(u)}{h_C(u)}}{u\in \RR^d\setminus\setn{0}}\\
&=\sup\setcond{\frac{h_{K-K}(u)}{h_C(u)}}{u\in B\setminus\setn{0}}\\
&=\sup\setcond{h_{K-K}\lr{\frac{u}{h_C(u)}}}{u\in B\setminus\setn{0}}\\
&=\sup\setcond{h_{K-K}(x)}{x\in C^\circ \setminus\setn{0}}\\
&=\sup\setcond{h_{K-K}(x)}{x\in C^\circ}\\
&=\sup\setcond{\gamma_C(x-y)}{x,y\in K}\\
&=R(K-K,C)\\
&=R(K-K,\frac{1}{2}(C-C))\\
&=\sup\setcond{\gamma_{\frac{1}{2}(C-C)}(x)}{x\in K-K}\\
&=2\sup\setcond{R(\setn{0,x},C)}{x\in K-K}\\
&=2\sup\setcond{R(\setn{x,y},C)}{x,y\in K}
\end{align*}
by using Theorem~\ref{thm:diameter_max_distance}, Proposition~\ref{lem:central_symmetry_at_all}, Lemma~\ref{lem:segment_length}, and Proposition~\ref{prop:circumradius}. Note that Lemma~\ref{lem:segment_length} is independent of centeredness of $C$ and, therefore, can be similarly used in the general case, which comes next. From now on, we do not assume $C=-C$. We apply the calculations for the symmetric case and obtain
\begin{align*}
&\norel 2\sup\setcond{\frac{h_{K-K}(u)}{h_{C-C}(u)}}{u\in \RR^d\setminus\setn{0}}\\
&=2\sup\setcond{\frac{h_{(K-K)-(K-K)}(u)}{h_{(C-C)-(C-C)}(u)}}{u\in \RR^d\setminus\setn{0}}\\
&=R\lr{(K-K)-(K-K),\frac{1}{2}((C-C)-(C-C))}\\
&=R\lr{K-K,\frac{1}{2}(C-C)}\\
&=2\sup\setcond{R(\setn{x,y},C)}{x,y\in K}\\
&=2\sup\setcond{R(\setn{0,x},C)}{x\in K-K}\\
&=\sup\setcond{R(\setn{-x,x},C)}{x\in K-K}\\
&\leq R(K-K,C)\\
&\leq \inf\setcond{\lambda >0}{K-K\subset \lambda C}\\
&=\sup_{x\in K-K}\gamma_C(x)\\
&=\sup_{x,y\in K}\gamma_C(x-y).
\end{align*}

In order to prove the addendum \eqref{eq:symmetric_diameter_addition}, let $K\in \CK^d$. Then
\begin{align*}
&\norel 2\sup\setcond{R(\setn{x,y},C)}{x,y\in K}\\
&=2\sup\setcond{\frac{\norm{x-y}}{l_C\lr{\frac{x-y}{\norm{x-y}}}}}{x,y\in K,x\neq y}\\
&=2\sup_{u\in \RR^d\setminus\setn{0}}\sup\setcond{\frac{\alpha}{l_C(u)}}{\alpha >0,\alpha u\in K-K}\\
&=2\sup_{u\in \RR^d\setminus\setn{0}}\frac{\sup\setcond{\alpha}{\alpha >0,\alpha u\in K-K}}{l_C(u)}\\
&=2\sup\setcond{\frac{l_K(u)}{l_C(u)}}{u\in \RR^d\setminus\setn{0}}.\qedhere
\end{align*}
\end{proof}
The following examples show that the inequalities in \eqref{eq:symmetric_diameter} need not be strict if $K$ and $C$ are not centrally symmetric but, on the other hand, can be strict even if $K$ is centrally symmetric. An illustration of these examples is provided by Figure~\ref{fig:strict_diameter_ineq}.
\begin{Bsp}\label{ex:strict_diameter_ineq}
\begin{enumerate}[label={(\alph*)},align=left]
\item{\label{strict_diameter_ineq_reuleaux}Let $d=2$ and
\begin{equation*}
C=-K=((2,0)+2\sqrt{3}B)\cap ((-1,\sqrt{3})+2\sqrt{3}B)\cap ((-1,-\sqrt{3})+2\sqrt{3}B).
\end{equation*}
Then $K-K=C-C=2\sqrt{3}B$, \dah
\begin{equation*}
2\sup\setcond{\frac{h_{K-K}(u)}{h_{C-C}(u)}}{u\in\RR^d\setminus\setn{0}}=2,
\end{equation*}
but $R(K-K,C)=\sup\setcond{\gamma_C(x-y)}{x,y\in K}=\frac{1}{2}(3+\sqrt{3})\approx 2.366025$.}
\item{\label{strict_diameter_ineq_triangle}Let $d=2$, $C=\co\setn{(2,0),(-1,\sqrt{3}),(-1,-\sqrt{3})}$, and
\begin{equation*}
K=\co\setn{(-\sqrt{3},-\sqrt{3}),(-\sqrt{3},\sqrt{3}),(\sqrt{3},-\sqrt{3}),(\sqrt{3},\sqrt{3})}.
\end{equation*}
Then
\begin{align*}
\sup\setcond{\gamma_C(x-y)}{x,y\in K}&=3+\sqrt{3}\approx 4.732,\\
R(K-K,C)&=2+\frac{4}{\sqrt{3}}\approx 4.3094,\\
2\sup\setcond{R(\setn{x,y},C)}{x,y\in K}&=\frac{2}{3}(3+\sqrt{3})\approx 3.1547,\\
2\sup\setcond{\frac{h_{K-K}(u)}{h_{C-C}(u)}}{u\in \RR^d\setminus\setn{0}}&=\frac{2}{3}(3+\sqrt{3})\approx 3.1547.
\end{align*}}
\end{enumerate}
\end{Bsp}
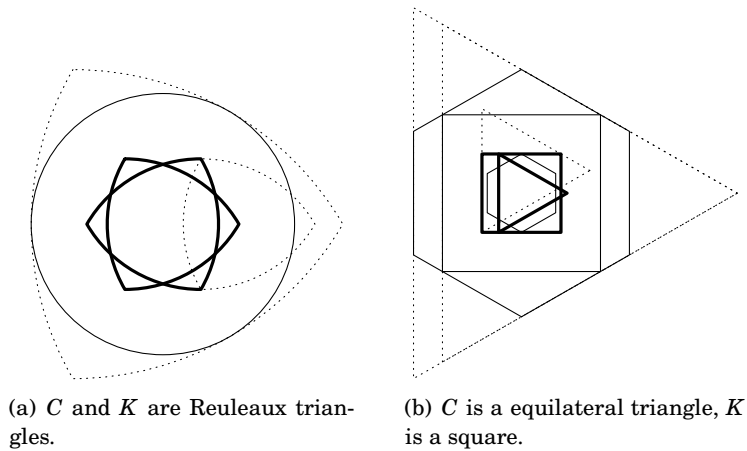
\begin{figure}[h!]
\begin{center}
\subfigure[$C$ and $K$ are Reuleaux triangles.]{
\begin{tikzpicture}[line cap=round,line join=round,>=stealth,x=0.5cm,y=0.5cm]
\pgfmathsetmacro{\a}{sqrt(3)}
\draw [line width=1.2pt,scale=-1](-1,\a) arc (150:210:2*\a) -- (-1,-\a) arc (270:330:2*\a) -- (2,0) arc (30:90:2*\a)--cycle;
\draw [line width=1.2pt](-1,\a) arc (150:210:2*\a) -- (-1,-\a) arc (270:330:2*\a) -- (2,0) arc (30:90:2*\a)--cycle;
\draw (0,0) circle (2*\a);
\draw[dotted,shift={(2,0)}] (-1,\a) arc (150:210:2*\a) -- (-1,-\a) arc (270:330:2*\a) -- (2,0) arc (30:90:2*\a)--cycle;
\draw[dotted,scale=0.5*(3+\a)] (-1,\a) arc (150:210:2*\a) -- (-1,-\a) arc (270:330:2*\a) -- (2,0) arc (30:90:2*\a)--cycle;
\end{tikzpicture}
}\qquad
\subfigure[$C$ is a equilateral triangle, $K$ is a square.]{
\begin{tikzpicture}[line cap=round,line join=round,>=stealth,x=0.3cm,y=0.3cm]
\pgfmathsetmacro{\a}{sqrt(3)}
\draw [line width=1.2pt](2,0)-- (-1,\a)-- (-1,-\a)--cycle;
\draw[dotted,shift={(2-2/\a,0)},scale=2+4/\a] (2,0)-- (-1,\a)-- (-1,-\a)--cycle;
\draw[dotted,shift={(1-2/\a,1)},scale=1+1/\a] (2,0)-- (-1,\a)-- (-1,-\a)--cycle;
\draw[line width=1.2pt] (-\a,\a)-- (\a,\a)-- (\a,-\a)-- (-\a,-\a)--cycle;
\draw[scale=2] (-\a,\a)-- (\a,\a)-- (\a,-\a)-- (-\a,-\a)--cycle;
\draw (0,-\a)-- (1.5,-0.5*\a)-- (1.5,0.5*\a)-- (0,\a)-- (-1.5,0.5*\a)-- (-1.5,-0.5*\a)--cycle;
\draw[scale=2+2/\a] (0,-\a)-- (1.5,-0.5*\a)-- (1.5,0.5*\a)-- (0,\a)-- (-1.5,0.5*\a)-- (-1.5,-0.5*\a)--cycle;
\draw[dotted,scale=3+\a](2,0)-- (-1,\a)-- (-1,-\a)--cycle;
\end{tikzpicture}
}
\end{center}\caption{Illustration of Example~\ref{ex:strict_diameter_ineq}: The sets $C$ and $K$ are depicted in bold lines.}\label{fig:strict_diameter_ineq}
\end{figure}
Note that usually the diameter is defined on the lines of Theorem~\ref{thm:diameter_max_distance}\ref{max_distance_diam}, see \cite{GonzalezHe2012,GonzalezHeHi2015,HenkHe2009} for the Euclidean case (\dah $C=B$) and \cite{GritzmannKl1992} for the normed case (\dah $C=-C\in\KK^d_0$). In the general setting, each of the representations may have its own benefits. However, following \cite[Definition~5.2]{BrandenbergKoe2014}, we can define the notion of diameter via circumradii of two-element subsets which is, by Lemma~\ref{lem:segment_length}, the usual diameter with respect to the norm generated by $\frac{1}{2}(C-C)$.
\begin{Def}\label{def:diameter}
The \emph{diameter} $K$ with respect to $C$ is
\begin{equation*}
D(K,C)=2\sup\setcond{R(\setn{x,y},C)}{x,y\in K}.
\end{equation*}
\end{Def}
The diameter also behaves nicely under hull operations and Minkowski sums in the first arguments, as well as under independent translations and scalings of both arguments.
\begin{Prop}\label{prop:diameter}
Let $K,K^\prime\subset \RR^d$, $C,C^\prime\in\CK^d$, and $\alpha,\beta>0$. Then we have
\begin{enumerate}[label={(\alph*)},align=left]
\item{$D(K^\prime,C^\prime)\leq D(K,C)$ if $K^\prime \subset K$ and $C\subset C^\prime$,\label{diameter_inclusion}}
\item{$D(K,C)=D(\cl(K),C)=D(\co(K),C)$,\label{diameter_closed_convex}}
\item{$D(K+K^\prime,C)\leq D(K,C)+D(K^\prime,C)$,\label{diameter_addition}}
\item{$D(x+K,y+C)=r(K,C)$ for all $x,y\in\RR^d$,\label{diameter_translation}}
\item{$D(\alpha K,\beta C)=\frac{\alpha}{\beta}D(K,C)$,\label{diameter_scaling}}
\item{$D(K,C^\prime)\leq D(K,C)D(C,C^\prime)$.\label{diameter_nesting}}
\end{enumerate}
\end{Prop}
\begin{proof}
Statement~\ref{diameter_inclusion} is a consequence of Proposition~\ref{prop:circumradius}\ref{circumradius_inclusion}. Clearly, we have $D(K,C)\leq D(\cl(K),C)$ and $D(K,C)\leq D(\co(K),C)$. Futhermore, we obtain
\begin{align*}
D(\co(K),C)&=\sup\setcond{R(\setn{x,y},C)}{x,y\in \co(K)}\\
&=\sup\setcond{R(\setn{0,z},C)}{z\in \co(K-K)}\\
&=\sup\setcond{R\lr{\setn{0,\sum_{i=1}^n \lambda_i x_i},C}}{\begin{matrix}n\in\NN,x_i\in K-K,\lambda_i\geq 0,\\i\in\setn{1,\ldots,n},\sum_{i=1}^n\lambda_i =1\end{matrix}}\\
&\leq\sup\setcond{R\lr{\sum_{i=1}^n \lambda_i\setn{0, x_i},C}}{\begin{matrix}n\in\NN,x_i\in K-K,\lambda_i\geq 0+,\\i\in\setn{1,\ldots,n},\sum_{i=1}^n\lambda_i =1\end{matrix}}\\
&\leq\sup\setcond{\sum_{i=1}^n \lambda_i R(\setn{0, x_i},C)}{\begin{matrix}n\in\NN,x_i\in K-K,\lambda_i \geq 0,\\i\in\setn{1,\ldots,n},\sum_{i=1}^n\lambda_i =1\end{matrix}}\\
&\leq\sup\setcond{\sum_{i=1}^n \lambda_i \sup\setcond{R(\setn{0, w},C)}{w \in K-K}}{\begin{matrix}n\in\NN,x_i\in K-K,\lambda_i \geq 0,\\i\in\setn{1,\ldots,n},\sum_{i=1}^n\lambda_i =1\end{matrix}}\\
&\leq\sup\setcond{R(\setn{0,w},C)}{w\in K-K}\sup\setcond{\sum_{i=1}^n \lambda_i}{\begin{matrix}n\in\NN,\lambda_i\geq 0,\\i\in\setn{1,\ldots,n},\sum_{i=1}^n\lambda_i =1\end{matrix}}\\
&=D(K,C)
\end{align*}
and
\begin{align*}
D(\cl(K),C)&=\sup\setcond{R(\setn{x,y},C)}{x,y\in \cl(K)}\\
&=\sup\setcond{R(\setn{x_i,y_i},C)}{x_i,y_i\in K,i\in \NN, x_i\to x, y_i\to y}\\
&\leq \sup\setcond{\sup\setcond{R(\setn{w,z},C)}{w,z\in K}}{x_i,y_i\in K,i\in \NN, x_i\to x, y_i\to y}\\
&=\sup\setcond{R(\setn{w,z},C)}{w,z\in K}.
\end{align*}
This yields claim~\ref{inradius_closed_convex}. In order to prove part~\ref{inradius_addition}, we observe that
\begin{align*}
D(K+K^\prime,C)&=\sup\setcond{R(\setn{x,y},C)}{x,y\in K+K^\prime}\\
&=\sup\setcond{R(\setns{w+w^\prime,z+z^\prime},C)}{w,z\in K, w^\prime,z^\prime\in K^\prime}\\
&\leq \sup\setcond{R(\setns{w+w^\prime,z+z^\prime,w+z,w^\prime+z^\prime},C)}{w,z\in K, w^\prime,z^\prime\in K^\prime}\\
&\leq \sup\setcond{R(\setn{w,z},C)+R(\setns{w^\prime,z^\prime},C)}{w,z\in K, w^\prime,z^\prime\in K^\prime}\\
&=\sup\setcond{R(\setn{w,z},C)}{w,z\in K}+\sup\setcond{R(\setns{w^\prime,z^\prime},C)}{w^\prime,z^\prime\in K^\prime}\\
&=D(K,C)+D(K^\prime,C).
\end{align*}
In order to prove \ref{diameter_nesting}, we use the representation $D(K,C)=\sup\setcond{\gamma_{C-C}(x)}{x\in K-K}$. Without loss of generality, we may assume that $K$ is bounded since otherwise $D(K,C)=D(K,C^\prime)=+\infty$ due to the positive homogeneity of Minkowski functionals. Furthermore, we assume that $0\in\inte(C)\cap\inte(C^\prime)$ due to \ref{diameter_translation}. We have
\begin{align*}
\norel D(K,C^\prime)&=\sup\setcond{\gamma_{C^\prime-C^\prime}(x)}{x\in K-K}\\
&=\sup\setcond{\gamma_{C^\prime-C^\prime}(\alpha u)}{u\in \bd(B),\alpha\in \clseg{0}{r_{K-K}(u)}}\\
&=\sup\setcond{\gamma_{C^\prime-C^\prime}(r_{K-K}(u)u)}{u\in \bd(B)}\\
&=\sup\setcond{r_{K-K}(u)\gamma_{C-C}(u)\frac{\gamma_{C^\prime-C^\prime}(u)}{\gamma_{C-C}(u)}}{u\in \bd(B)}\\
&=\sup\setcond{r_{K-K}(u)\gamma_{C-C}(u)r_{C-C}(u)\gamma_{C^\prime-C^\prime}(u)}{u\in \bd(B)}\\
&\leq \sup\setcond{r_{K-K}(u)\gamma_{C-C}(u)}{u\in \bd(B)}\\
&\qquad \cdot \sup\setcond{r_{C-C}(u)\gamma_{C^\prime-C^\prime}(u)}{u\in \bd(B)}\\
&= \sup\setcond{\gamma_{C-C}(x)}{x\in K-K} \sup\setcond{\gamma_{C^\prime-C^\prime}(x)}{x\in C-C}\\
&=D(K,C)D(C,C^\prime).\qedhere
\end{align*}
\end{proof}

Finally, we remark that a classical upper bound of the diameter in terms of the circumradius is still valid in generalized Minkowski spaces. Namely $D(K,C)\leq 2R(K,C)$  for all $K\subset\RR^d$ and $C\in\KK^d_0$ with $0\in\inte(C)$, with equality if, \zB $C=-C$ and $K=-K$. This is follows immediately from Proposition~\ref{prop:circumradius}\ref{circumradius_inclusion} and Theorem~\ref{thm:symmetric_diameter}.

\subsection{Minimum Width}
In Euclidean space, the notion of \emph{minimum width} is intimately related to the notion of \emph{diameter}. The latter is the maximum of the width function (see Theorem~\ref{thm:diameter_max_distance}), the former is classically defined as the corresponding infimum. Here the reference to the (possibly non-centered) \enquote{unit ball} is done by considering the ratio of the width functions. At first, we collect relations between several representations of minimum width in normed spaces \cite[Theorem~3]{Averkov2003b} and within the general setting.
\begin{Lem}
Let $K\subset \RR^d$, $C\in \KK^d_0$. If $C=-C$, we have
\begin{equation}
2\inf\setcond{\frac{h_{K-K}(u)}{h_{C-C}(u)}}{u\in\RR^d\setminus\setn{0}}=\inf\setcond{\frac{h_{K-K}(u)}{\gamma_C^\circ(u)}}{u\in\RR^d\setminus\setn{0}}.\label{eq:width_representation2}
\end{equation}
In other words, the minimal ratio of the (Euclidean) width functions is equal to the minimal distance of parallel supporting hyperplanes of $K$, measured by the norm $\gamma_C$. If $h_{K-K}(u)>0$ for all $u\in\RR^d\setminus\setn{0}$, also the reverse implication is true.
\end{Lem}
\begin{proof}
The first statement is clear by the relations $\gamma_C^\circ=h_C$, $C-C=2C$, and Lemma~\ref{lem:support_function}\ref{support_function_scaling}. For the reverse statement, assume that $h_{K-K}(u)>0$ for all $u\in\RR^d\setminus\setn{0}$ and
\begin{equation}
2\inf\setcond{\frac{h_{K-K}(u)}{h_{C-C}(u)}}{u\in\RR^d\setminus\setn{0}}<\inf\setcond{\frac{h_{K-K}(u)}{h_C(u)}}{u\in\RR^d\setminus\setn{0}}.\label{eq:width_representation_fails}
\end{equation}
Then
\begin{equation*}
2\frac{h_{K-K}(u)}{h_{C-C}(u)}<\frac{h_{K-K}(u)}{h_C(u)}
\end{equation*}
for all $u\in\RR^d\setminus\setn{0}$, which is equivalent to
\begin{equation*}
\frac{h_{C-C}}{2}<h_C
\end{equation*}
or $h_C>h_{-C}$. Since $C$ is convex, this means $-C\subsetneq C$ which is impossible. We obtain the same result if we assume the reverse inequality in \eqref{eq:width_representation_fails}.
\end{proof}
\begin{Lem}
If $K,C\in \KK^d_0$, then
\begin{equation*}
r(K-K,C)=R(C,K-K)^{-1}=\lr{\sup\setcond{\skpr{u}{x}}{u\in (K-K)^\circ, x\in C}}^{-1}.
\end{equation*}
\end{Lem}
\begin{proof}
Note that $K-K$ is centered and apply Theorem~\ref{thm:symmetric_diameter}.
\end{proof}
\begin{Bem}
The claim of the previous lemma fails if $K$ is not convex. For example, if $K$ is a finite set, then $r(K-K,C)=0$. But 
\begin{align*}
(\co(K)-\co(K))^\circ&=(\co(K-K))^\circ=\setcond{y\in\RR^d}{h_{\co(K-K)}(y)\leq 1}\\
&=\setcond{y\in\RR^d}{h_{K-K}(y)\leq 1}=(K-K)^\circ,
\end{align*}
where we only used Lemma~\ref{lem:support_function}\ref{support_function_closed_convex}. It follows that 
\begin{equation*}
\sup\setcond{\skpr{u}{x}}{u\in (K-K)^\circ, x\in C}=\sup\setcond{\skpr{u}{x}}{u\in (\co(K)-\co(K))^\circ, x\in C},
\end{equation*}
which is apparently not equal to zero in general.
\end{Bem}
\begin{Lem}\label{lem:inradius_vs_brandenberg}
Let $K\in \KK^d$, $C\in \KK^d_0$. Then
\begin{equation*}
r(K-K,C)\leq 2\inf\setcond{\frac{h_{K-K}(u)}{h_{C-C}(u)}}{u\in\RR^d\setminus\setn{0}}
\end{equation*}
with equality if $C=-C$.
\end{Lem}
\begin{proof}
For all $\alpha <r(K-K,C)$, there exists $z\in\RR^d$ such that $z+\alpha C\subset K-K$. Using Lemma~\ref{lem:width_function}, we obtain
\begin{align*}
&&w_{z+\alpha C}(u)&\leq w_{K-K}(u) \text{ for all }u\in\RR^d\setminus\setn{0}\\
&\Longleftrightarrow& w_{\alpha C}(u)&\leq w_{K-K}(u) \text{ for all }u\in\RR^d\setminus\setn{0}\\
&\Longleftrightarrow& \alpha w_C(u)&\leq 2 h_{K-K}(u) \text{ for all }u\in\RR^d\setminus\setn{0}\\
&\Longleftrightarrow& \alpha &\leq 2\frac{h_{K-K}(u)}{h_{C-C}(u)} \text{ for all }u\in\RR^d\setminus\setn{0}.
\end{align*}
Passing $\alpha$ to $r(K-K,C)$, we obtain \eqref{eq:width_representation2}. Now let $C=-C$ and assume that \eqref{eq:width_representation2} is a strict inequality, \dah there exists $\alpha$ such that
\begin{equation*}
r(K-K,C)<\alpha<2\inf\setcond{\frac{h_{K-K}(u)}{h_{C-C}(u)}}{u\in\RR^d\setminus\setn{0}}.
\end{equation*}
Like above, we obtain $w_{\alpha C}< w_{K-K}$. Dividing by $2$, we have $h_{\alpha C}< h_{K-K}$. It follows that $r(K-K,C)C\subsetneq \alpha C\subsetneq K-K$. This is a contradiction to the definition of $r(K-K,C)$.
\end{proof}
\begin{Bem}\label{bem:strict_width_ineq}
If $C\neq -C$, we may have a strict inequality in \eqref{eq:width_representation2}. In the situation of Example~\ref{ex:strict_diameter_ineq}\ref{strict_diameter_ineq_reuleaux}, we obtain
\begin{equation*}
2\inf\setcond{\frac{h_{K-K}(u)}{h_{C-C}(u)}}{u\in\RR^d\setminus\setn{0}}=2,
\end{equation*}
but obviously $r(K-K,C)=\sqrt{3}\neq 2$, see Figure~\ref{fig:strict_width_ineq}.
\end{Bem}
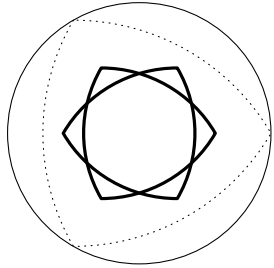
\begin{figure}[h!]
\begin{center}
\begin{tikzpicture}[line cap=round,line join=round,>=stealth,x=0.5cm,y=0.5cm]
\pgfmathsetmacro{\a}{sqrt(3)}
\draw [line width=1.2pt,scale=-1](-1,\a) arc (150:210:2*\a) -- (-1,-\a) arc (270:330:2*\a) -- (2,0) arc (30:90:2*\a)--cycle;
\draw [line width=1.2pt](-1,\a) arc (150:210:2*\a) -- (-1,-\a) arc (270:330:2*\a) -- (2,0) arc (30:90:2*\a)--cycle;
\draw[dotted,scale=\a] (-1,\a) arc (150:210:2*\a) -- (-1,-\a) arc (270:330:2*\a) -- (2,0) arc (30:90:2*\a)--cycle;
\draw(0,0) circle (2*\a);
\end{tikzpicture}
\end{center}\caption{Illustration of Remark~\ref{bem:strict_width_ineq}: $K$ and $C$ are Reuleaux triangles (bold lines).\label{fig:strict_width_ineq}}
\end{figure}
Summarizing, we obtain the following theorem on the notion of minimum width in normed spaces.
\begin{Satz}[{\cite[Theorem~3]{Averkov2003b}}]\label{thm:symmetric_width}
For $K\subset \CK^d$ and $C\in \KK^d_0$ with $C=-C$, the following numbers are equal:
\begin{enumerate}[label={(\alph*)},align=left]
\item{$r(K-K,C)$,}
\item{$2\inf\setcond{\frac{h_{K-K}(u)}{h_{C-C}(u)}}{u\in\RR^d\setminus\setn{0}}$,}
\item{$\inf\setcond{\frac{h_{K-K}(u)}{\gamma_C^\circ(u)}}{u\in\RR^d\setminus\setn{0}}$,}
\item{$(\sup\setcond{\skpr{u}{x}}{u\in (K-K)^\ast, x\in C})^{-1}$.}
\end{enumerate}
\end{Satz}
\begin{proof}
This is a combination of the previous lemmas. If $h_{K-K}\equiv +\infty$, then $K=\RR^d$ and all the numbers equal $+\infty$. Similarly, if there is $u\in\RR^d\setminus\setn{0}$ such that $h_{K-K}(u)=0$, then all the numbers equal $0$. (For the last item use the convention $1/0=+\infty$, $1/(+\infty)=0$.)
\end{proof}
Since we focus on containment problems (that is, finding in some sense extremal scaling factors), it is useful to take the minimal ratio of support functions as definition of minimum width.
\begin{Def}[{\cite[Definition~2.6]{BrandenbergKoe2014}}]\label{def:width}
The \emph{minimum width} of $K$ with respect to $C$ is
\begin{align*}
\omega(K,C)&\defeq 2\inf\setcond{\frac{h_{K-K}(u)}{h_{C-C}(u)}}{u\in\RR^d\setminus\setn{0}}\\
&=2\inf\setcond{R(K,C+L)}{L\in\LL^d_{d-1}},
\end{align*}
where $\LL^d_{d-1}$ denotes the family of $(d-1)$-dimensional linear subspaces of $\RR^d$.
\end{Def}
\begin{Prop}\label{prop:width}
Let $K,K^\prime\subset \RR^d$, $C,C^\prime\in\CK^d$, and $\alpha,\beta>0$. Then we have
\begin{enumerate}[label={(\alph*)},align=left]
\item{$\omega(K^\prime,C^\prime)\geq \omega(K,C)$ if $K^\prime \subset K$ and $C\subset C^\prime$,\label{width_inclusion}}
\item{$\omega(K,C)=\omega(\cl(K),C)$ if $K$ is convex,\label{width_closed_convex}}
\item{$\omega(K+K^\prime,C)\geq \omega(K,C)+\omega(K^\prime,C)$,\label{width_addition}}
\item{$\omega(x+K,y+C)=\omega(K,C)$ for all $x,y\in\RR^d$,\label{width_translation}}
\item{$\omega(\alpha K,\beta C)=\frac{\alpha}{\beta}\omega(K,C)$,\label{width_scaling}}
\item{$\omega(K,C^\prime)\geq \omega(K,C)\omega(C,C^\prime)$.\label{width_nesting}}
\end{enumerate}
\end{Prop}
\begin{proof}
For $x\in\RR^d$, $\lambda\geq 0$, and $L\in\LL^d_{d-1}$, we have $\cl(K)\subset x+L+\lambda C\Longleftrightarrow K\subset x+L+\lambda C$. This proves \ref{width_closed_convex}.
In order to prove part~\ref{width_addition}, we observe
\begin{align*}
&\norel\omega(K+K^\prime,C)\\
&=\inf\setcond{\frac{h_{K+K^\prime-K-K^\prime}(u)}{h_{C-C}(u)}}{u\in\RR^d\setminus\setn{0}}\\
&=\inf\setcond{\frac{h_{K-K}(u)}{h_{C-C}(u)}+\frac{h_{K^\prime-K^\prime}(u)}{h_{C-C}(u)}}{u\in\RR^d\setminus\setn{0}}\\
&\geq \inf\setcond{\frac{h_{K-K}(u)}{h_{C-C}(u)}}{u\in\RR^d\setminus\setn{0}}+\inf\setcond{\frac{h_{K^\prime-K^\prime}(u)}{h_{C-C}(u)}}{u\in\RR^d\setminus\setn{0}}\\
&=\omega(K,C)+\omega(K^\prime,C).
\end{align*}
Finally, we have
\begin{align*}
&\norel \omega(K,C^\prime)\\
&=\inf\setcond{\frac{h_{K-K}(u)}{h_{C-C}(u)}\frac{h_{C-C}(u)}{h_{C^\prime-C^\prime}(u)}}{u\in\RR^d\setminus\setn{0}}\\
&\geq\inf\setcond{\frac{h_{K-K}(u)}{h_{C-C}(u)}}{u\in\RR^d\setminus\setn{0}}\inf\setcond{\frac{h_{C-C}(u)}{h_{C^\prime-C^\prime}(u)}}{u\in\RR^d\setminus\setn{0}}\\
&=\omega(K,C^\prime)\omega(C,C^\prime).\qedhere
\end{align*}
\end{proof}

\section{Open questions}\label{chap:open}
The increasing interest in vector spaces equipped with Minkowski functionals can be observed in various directions. For example, \emph{gauges} or \emph{convex distance functions} occur in computational geometry, operations research, and location science (see, \zB \cite{HeMaWu2013,PlastriaCa2001,Ma2000,IckingKlMaNiWe2001,Santos1996,JahnKuMaRi2014}. In the present paper, a gentle start is provided for applying this setting to basic metrical notions of convex geometry. Various further natural questions occur immediately. For example, are there reverse inequalities for Proposition~\ref{prop:circumradius}\ref{circumradius_addition} and Proposition~\ref{prop:inradius}\ref{inradius_addition} like in \cite[Theorems~1.1,~1.2]{GonzalezHe2012}? How can we use these quantities to obtain generalizations of diametrical maximality \cite{MorenoSc2012a, MorenoSc2012c}, constant width \cite[\textsection~2]{MartiniSw2004}, and reducedness \cite{LassakMa2014}?

\providecommand{\bysame}{\leavevmode\hbox to3em{\hrulefill}\thinspace}
\providecommand{\MR}{\relax\ifhmode\unskip\space\fi MR }
\providecommand{\MRhref}[2]{%
  \href{http://www.ams.org/mathscinet-getitem?mr=#1}{#2}
}
\providecommand{\href}[2]{#2}


\begin{thebibliography}{10}

\bibitem{AlonsoMaSp2012b}
J.~Alonso, H.~Martini, and M.~Spirova, \emph{{M}inimal enclosing discs, circumcircles, and circumcenters in normed planes ({P}art {II})}, Comput. Geom. \textbf{45} (2012), no.~7, pp.~350--369,
  \href{http://dx.doi.org/10.1016/j.comgeo.2012.02.003}{doi:~10.1016/j.comgeo.2012.02.003}.

\bibitem{Averkov2003b}
G.~Averkov, \emph{{O}n cross-section measures in {M}inkowski spaces}, Extracta Math. \textbf{18} (2003), no.~2, pp.~201--208.

\bibitem{BauschkeCo2011}
H.~H. Bauschke and P.~L. Combettes, \emph{{C}onvex {A}nalysis and {M}onotone {O}perator {T}heory in {H}ilbert {S}paces}, CMS Books in Mathematics, Springer, New York, 2011,
  \href{http://dx.doi.org/10.1007/978-1-4419-9467-7}{doi:~10.1007/978-1-4419-9467-7}.

\bibitem{BonnesenFe1987}
T.~Bonnesen and W.~Fenchel, \emph{{T}heory of {C}onvex {B}odies}, BCS Associates, Moscow, ID, 1987, translated from the {G}erman and edited by {L}.
  {B}oron, {C}. {C}hristenson and {B}. {S}mith.

\bibitem{BrandenbergKoe2013}
R.~Brandenberg and S.~K{\"{o}}nig, \emph{{N}o dimension-independent core-sets for containment under homothetics}, Discrete Comput. Geom. \textbf{49} (2013), no.~1, pp.~3--21,
  \href{http://dx.doi.org/10.1007/s00454-012-9462-0}{doi:~10.1007/s00454-012-9462-0}.

\bibitem{BrandenbergKoe2014}
\bysame, \emph{{S}harpening geometric inequalities using computable symmetry measures}, Mathematika (2014),
  \href{http://dx.doi.org/10.1112/S0025579314000291}{doi:~10.1112/S0025579314000291}.

\bibitem{BrandenbergRo2011}
R.~Brandenberg and L.~Roth, \emph{{M}inimal containment under homothetics: {A} simple cutting plane approach}, Comput. Optim. Appl. \textbf{48} (2011), no.~2, pp.~325--340,
  \href{http://dx.doi.org/10.1007/s10589-009-9248-3}{doi:~10.1007/s10589-009-9248-3}.

\bibitem{GonzalezHe2012}
B.~Gonz{\'{a}}lez and M.~Hern{\'{a}}ndez~Cifre, \emph{{S}uccessive radii and {M}inkowski addition}, Monatsh. Math. \textbf{166} (2012), no.~3-4, pp.~395--409,
  \href{http://dx.doi.org/10.1007/s00605-010-0268-y}{doi:~10.1007/s00605-010-0268-y}.

\bibitem{GonzalezHeHi2015}
B.~Gonz{\'{a}}lez, M.~Hern{\'{a}}ndez~Cifre, and A.~Hinrichs, \emph{{S}uccessive radii of families of convex bodies}, Bull. Aust. Math. Soc. \textbf{91} (2015), no.~2, pp.~331--344,
  \href{http://dx.doi.org/10.1017/S0004972714000902}{doi:~10.1017/S0004972714000902}.

\bibitem{GritzmannKl1992}
P.~Gritzmann and V.~Klee, \emph{{I}nner and outer {$j$}-radii of convex bodies in finite-dimensional normed spaces}, Discrete Comput. Geom. \textbf{7} (1992), no.~1, pp.~255--280,
  \href{http://dx.doi.org/10.1007/BF02187841}{doi:~10.1007/BF02187841}.

\bibitem{HeMaWu2013}
C.~He, H.~Martini, and S.~Wu, \emph{{O}n bisectors for convex distance functions}, Extracta Math. \textbf{28} (2013), no.~1, pp.~57--76.

\bibitem{HenkHe2009}
M.~Henk and M.~Hern{\'{a}}ndez~Cifre, \emph{{S}uccessive minima and radii}, Canad. Math. Bull. \textbf{52} (2009), no.~3, pp.~380--387,
  \href{http://dx.doi.org/10.4153/CMB-2009-041-2}{doi:~10.4153/CMB-2009-041-2}.

\bibitem{IckingKlMaNiWe2001}
C.~Icking, R.~Klein, L.~Ma, S.~Nickel, and A.~Wei{\ss}ler, \emph{{O}n bisectors for different distance functions}, Discrete Appl. Math. \textbf{109} (2001), no.~1-2, pp.~139--161,
  \href{http://dx.doi.org/10.1016/S0166-218X(00)00238-9}{doi:~10.1016/S0166-218X(00)00238-9}.

\bibitem{Jahn2015a}
T.~Jahn, \emph{{G}eometric algorithms for minimal enclosing discs in strictly convex normed spaces}, 2015,
  \href{http://arxiv.org/abs/1410.4725v1}{arxiv:~1410.4725v1}.

\bibitem{Jahn2015c}
\bysame, \emph{{S}uccessive radii and ball operators in generalized {M}inkowski spaces}, 2015, \href{http://arxiv.org/abs/1411.1286v2}{arxiv:~1411.1286v2}.

\bibitem{JahnKuMaRi2014}
T.~Jahn, Y.~S. Kupitz, H.~Martini, and C.~Richter, \emph{{M}insum location extended to gauges and to convex sets}, (2014),
  \href{http://dx.doi.org/10.1007/s10957-014-0692-6}{doi:~10.1007/s10957-014-0692-6}.

\bibitem{LassakMa2014}
M.~Lassak and H.~Martini, \emph{{R}educed convex bodies in finite-dimensional normed spaces: {A} survey}, Results Math. \textbf{66} (2014), no.~3-4, pp.~405--426,
  \href{http://dx.doi.org/10.1007/s00025-014-0384-4}{doi:~10.1007/s00025-014-0384-4}.

\bibitem{Ma2000}
L.~Ma, \emph{{B}isectors and {V}oronoi {D}iagrams for {C}onvex {D}istance {F}unctions}, Ph.D. thesis, Fernuniversit{\"{a}}t Hagen, 2000.

\bibitem{MartinMaSp2014}
P.~Mart{\'{\i}}n, H.~Martini, and M.~Spirova, \emph{{C}hebyshev sets and ball operators}, J. Convex Anal. \textbf{21} (2014), no.~3, pp.~601--618.

\bibitem{MartiniSw2004}
H.~Martini and K.~Swanepoel, \emph{{T}he geometry of {M}inkowski spaces -- a survey, {P}art {II}}, Expo. Math. \textbf{22} (2004), no.~2, pp.~93--144,
 \href{http://dx.doi.org/10.1016/S0723-0869(04)80009-4}{doi:~10.1016/S0723-0869(04)80009-4}.

\bibitem{MorenoSc2012a}
J.~P. Moreno and R.~Schneider, \emph{{D}iametrically complete sets in {M}inkowski spaces}, Israel J. Math. \textbf{191} (2012), no.~2, pp.~701--720,
  \href{http://dx.doi.org/10.1007/s11856-012-0003-6}{doi:~10.1007/s11856-012-0003-6}.

\bibitem{MorenoSc2012c}
\bysame, \emph{{S}tructure of the space of diametrically complete sets in a {M}inkowski space}, Discrete Comput. Geom. \textbf{48} (2012), no.~2, pp.~467--486,
  \href{http://dx.doi.org/10.1007/s00454-011-9393-1}{doi:~10.1007/s00454-011-9393-1}.

\bibitem{PlastriaCa2001}
F.~Plastria and E.~Carrizosa, \emph{{G}auge distances and median hyperplanes}, J. Optim. Theory Appl. \textbf{110} (2001), no.~1, pp.~173--182,
 \href{http://dx.doi.org/10.1023/A:1017551731021}{doi:~10.1023/A:1017551731021}.

\bibitem{Santos1996}
F.~Santos, \emph{{O}n {D}elaunay oriented matroids for convex distance functions}, Discrete Comput. Geom. \textbf{16} (1996), no.~2, pp.~197--210, \href{http://dx.doi.org/10.1007/BF02716807}{doi:~10.1007/BF02716807}.

\bibitem{Sylvester1857}
J.~J. Sylvester, \emph{{A} question in the geometry of situation}, Q. J. Math. \textbf{1} (1857), p.~79.
\end{thebibliography}
\end{document}